\documentclass[11pt]{amsart}

\usepackage{amssymb,mathtools,color,tikz-cd}
\usepackage[letterpaper,textwidth=5.5in,textheight=8.5in,centering]{geometry}
\usepackage[hidelinks]{hyperref}

\title{Algebraic geometric framework of Rogers--Ramanujan identities}

\author{Yifeng Huang}
\address{Department of Mathematics, Suite W401,
400 Dowman Drive, Emory University,
Atlanta, GA 30322}\email{yifeng.huang@emory.edu}
\author{Kenny Lau}
\address{Axiom Math, 124 University Avenue, Palo Alto, CA 94301} \email{kenny@axiommath.ai}
\author{Ken Ono}
 \address{Axiom Math, 124 University Avenue, Palo Alto, CA 94301}\email{ken@axiommath.ai}
\author{Peter Paule}
\address{Research Institute for Symbolic Computation (RISC),
Johannes Kepler University, A-4040 Linz, Austria,
and Center for Applied Mathematics (TCAM),
Tianjin University, Tianjin 300072, China}\email{Peter.Paule@risc.jku.at}
\date{}
\newtheorem{theorem}{Theorem}
\newtheorem{prop}[theorem]{Proposition}

\newtheorem{lemma}[theorem]{Lemma}

\newtheorem{conjecture}[theorem]{Conjecture}

\newtheorem{remark}[theorem]{Remark}
\theoremstyle{remark}
\newtheorem*{remark*}{Remark}
\theoremstyle{definition}

\newcommand{\qPochN}[3]{\left(#1;#2\right)_{#3}}        
\newcommand{\qbinom}[2]{\genfrac{[}{]}{0pt}{}{#1}{#2}_{\!q}} 

\newcommand{\Z}{\mathbb{Z}}
\newcommand{\Q}{\mathbb{Q}}
\newcommand{\R}{\mathbb{R}}

\newcommand{\Fq}{\mathbb{F}_q}
\newcommand{\defn}{\textbf}

\DeclarePairedDelimiter{\abs}{\lvert}{\rvert}
\DeclarePairedDelimiter{\set}{\{}{\}}
\usepackage{xspace}
\def\huang{the first author\xspace}
\def\lau{the second author\xspace}
\def\ono{the fourth author\xspace}

\newcommand{\dast}{\mathbin{{\ast}{\ast}}}

\newcommand{\qbin}[3]{\genfrac{[}{]}{0pt}{}{#1}{#2}_{#3}}

\def\KK{\mathbb K}

\newcounter{mmacnt}
\def\restartmma{\setcounter{mmacnt}{0}}
\restartmma \catcode`|=\active
\def|#1|{\mathrm{#1}}
\catcode`|=12
\newenvironment{mma}{
  \par
 \catcode`|=\active
 \parskip=2pt\parindent=0pt 
 \small
 \def\In##1\\{%
   \def\linebreak{\hfill\break\null\Quad}%
   \refstepcounter{mmacnt}
   \hangindent=2.5em\hangafter=0
   \leavevmode
   \llap{\tiny\sffamily In[\arabic{mmacnt}]:=\kern.5em}%
   \mathversion{bold}\scriptsize$\tt\bf\displaystyle##1$\normalsize
   \mathversion{normal}\par
 }%
 \def\Print##1\\{%
   \def\linebreak{\hfill\break}%
   \hangindent=2.5em\hangafter=0
   \leavevmode\scriptsize ##1\par}%
 \def\Out##1\\{%
   \vspace*{-0.2cm}\def\linebreak{$\hfill\break\null\hfill$}%
   \kern\abovedisplayskip\par
   \hangindent=2.5em\hangafter=0
   \leavevmode
   \llap{\tiny\sffamily Out[\arabic{mmacnt}]=\kern.5em}
   \scriptsize$\displaystyle\tt##1$\normalsize\hfill\null\par
   \kern\belowdisplayskip\vspace*{-0.3cm}
 }%
 \def\Warning##1##2\\{%
   \def\linebreak{\hfill\break}%
   \hangindent=2.5em\hangafter=0
   \leavevmode
   {\scriptsize##1 : ##2}\par}%
}{%
 \par\smallskip
}

\newcommand{\LoadP}[1]{\fcolorbox{black}{white}{
\begin{minipage}[t]{12.5cm}
\footnotesize #1
\end{minipage}}}

\begin{document}

\begin{abstract}
The Rogers--Ramanujan identities equate a $q$-series whose exponents are
governed by a quadratic form with an infinite product supported on two
residue classes modulo~$5$. Identities of this shape are scarce, and a
central problem is to identify the structures that produce them in
families. Huang, Jiang, and Oblomkov have proposed a source of a new
kind: to each pair of coprime integers $a,b>1$ they attach an
infinite-rank $q$-series $Z_{a,b}(q)$, assembled from counts of
commuting nilpotent matrix pairs $(A,B)$ with $A^a=B^b$ over finite
fields, and they conjecture that it equals an explicit product of
$(a-1)(b-1)/2$ modular units of level $a+b$. The $a=2$ cases are the
Andrews--Gordon identities; no case with $a>2$ was known. We prove the
conjecture for $(a,b)=(3,4)$, $(3,5)$, $(3,7)$, and $(3,8)$. Our proofs
pass through a finer sum-to-sum identity, which we conjecture for all
$b$ coprime to $3$ and establish for all $b$ when $q=1$. Lau and Ono
have since proved that identity in general, and with it the full $a=3$
case. These identities have been formalized and verified in Lean by AxiomProver.
\end{abstract}



\maketitle

\section{Introduction}

The Rogers--Ramanujan (RR) identities
\begin{equation}\label{G}
G(q):=\sum_{n=0}^\infty\frac{q^{n^2}}{(1-q)\cdots(1-q^n)}
=\prod_{n=0}^\infty\frac{1}{(1-q^{5n+1})(1-q^{5n+4})}
\end{equation}
and
\begin{equation}\label{H}
H(q):=\sum_{n=0}^\infty\frac{q^{n^2+n}}{(1-q)\cdots(1-q^n)}
=\prod_{n=0}^\infty\frac{1}{(1-q^{5n+2})(1-q^{5n+3})}
\end{equation}
are among the most consequential identities in mathematics. Read off
coefficientwise, \eqref{G} asserts that for every $n$ the number of
partitions of $n$ into parts that differ by at least $2$ equals the
number of partitions of $n$ into parts congruent to $\pm 1 \pmod 5$,
while \eqref{H} asserts the same for partitions whose parts differ by at
least $2$ and exceed $1$, matched against partitions into parts
congruent to $\pm 2 \pmod 5$. A \emph{gap} condition on one side, a
\emph{congruence} condition on the other. Constraints of these two kinds
have no evident reason ever to precisely match for all $n.$

The identities are also statements about modular objects, and this is
the source of their depth. Both sides of \eqref{G} and \eqref{H} are, up
to a rational power of $q$, modular functions of level $5$, and their ratio $H(q)/G(q)$ is the
Rogers--Ramanujan continued fraction, whose special values Ramanujan
evaluated to remarkable effect \cite[\S 11]{PauleRadu2018}. An identity
of this kind therefore asserts that an object with no visible
modularity, a sum whose exponents come from a quadratic form, is in
fact modular, and to prove one is to exhibit the hidden structure
responsible.

The identities are not isolated. They are the first members of the
Andrews--Gordon identities, an infinite family in which the modulus $5$
is replaced by an arbitrary odd modulus $2k+3$ and the gap condition is
imposed on windows of $k$ consecutive parts. Analogues attached to
higher-rank root systems have since been produced by Bailey-type
machinery, among them the $\mathrm{A}_2$ Rogers--Ramanujan identities of
Andrews, Schilling, and Warnaar \cite{ASW1999} and the $\mathrm{A}_2$
Andrews--Gordon identities of Warnaar \cite{WarnaarAG}. Every such
family has the same silhouette: on one side a multiple $q$-series whose
exponents are values of a positive definite quadratic form and whose
summand is a product of Gaussian binomial coefficients, and on the other
a product of modular units. Families of this shape are rare, and each
one discovered so far has been connected to rich structure, such as
characters of modules over affine Lie algebras, order parameters of
exactly solvable lattice models, and chains of Bailey pairs. The governing
question is therefore less how to prove any single identity than what
sort of mathematical structure generates such a family at all.

\medskip

On the other hand, numerical semigroups with two generators are
classical objects in mathematics. For $a,b>1$ with $\gcd(a,b)=1$, let
$\Gamma=\langle a,b\rangle$ be the sub-semigroup of
$\mathbb{N}=\{0,1,\dots\}$ generated by $a$ and $b$, and let
$G=\mathbb{N}\setminus \Gamma$ be the set of gaps, of cardinality
$\abs{G}=(a-1)(b-1)/2$. It is a beautiful theorem of Sylvester in 1883
that the Frobenius element $f=\max G$ has the closed-form formula
$f=ab-a-b$. These numerical semigroups have rich combinatorics, algebra
and geometry. What will be relevant here is the poset structure on $G$
defined by $i\preceq j$ if and only if $j-i\in \Gamma$, and the map
$g:Y\to G$, where
\[ Y=\set{(x,y):x,y\in \Z_{\geq 1}, ab-ax-by>0},\quad g(x,y):=ab-ax-by.\]
It is known to Sylvester that $g$ is a bijection. We denote the inverse
images of $g$ by $(x(i),y(i)):=g^{-1}(i)$, called the Sylvester
coordinates of $i$. We have $i\preceq j$ if and only if $x(i)\geq x(j)$
and $y(i)\geq y(j)$. This bijection identifies $G$ with the Young
diagram $Y$ consisting of boxes in an $a\times b$ rectangle strictly on
one side of the diagonal, and identifies $\preceq$ with containment of
boxes. This viewpoint makes connections to algebraic combinatorics
natural; for example, sub-Young-diagrams of $Y$ are in natural bijection
with $a\times b$ rectangular Dyck paths, leading to connections to
$q,t$-Catalan numbers and to the geometry of the Hilbert scheme of
points on the singular curve $x^a=y^b$.

\medskip

These two stories meet by way of a source of Rogers--Ramanujan type
identities of a genuinely new kind: the enumerative geometry of plane curve singularities. For
coprime $a,b>1$, the \defn{torus knot singularity} $x^a=y^b$ has been
studied intensively through its Hilbert schemes of points, whose
generating function of point counts is the $a\times b$ rational
$q,t$-Catalan number and, by the Oblomkov--Rasmussen--Shende conjecture,
encodes the HOMFLY-PT invariant and the Khovanov--Rozansky homology of
the associated algebraic link \cite{GorskyMazin,ORS2018}. Replacing the
Hilbert scheme by the \defn{Quot scheme of points}, and then by the
\defn{stack of zero-dimensional coherent sheaves}, yields rank-$n$ and
rank-$\infty$ refinements of this picture
\cite{Huang2023,HuangJiang,Huang2024}. 
The rank-$n$ invariant is a lattice zeta function in the sense of Solomon \cite{Solomon1977}, counting finite-index submodules of a rank-$n$ torsion-free module. 
The rank-$\infty$ invariant is a Cohen--Lenstra type series counting finite modules weighted by the orders of their automorphism groups \cite{CohenLenstra}, and equivalently a count of matrix points in the sense of \cite{HOS2023}. 
Computations in \cite{Huang2023,HuangJiang} exposed a
phenomenon with no rank-one shadow: the rank-$\infty$ series appear to
evaluate to Rogers--Ramanujan type infinite products, with the rank-$n$
series serving as their finitizations. This is the state of affairs that
motivates the present work.

\medskip

\subsection{Can Rogers--Ramanujan type identities arise from the algebraic combinatorics of numerical semigroups?}

\medskip

Recently, \huang, Jiang, and Oblomkov \cite{HJO} defined a high-rank generalization of the rational $q$-Catalan number. More precisely, for each $n\geq 1$, they define a polynomial with nonnegative integer coefficients in $q$, called the rank-$n$ $a\times b$ rational $q$-Catalan number, that realizes the usual $a\times b$ rational $q$-Catalan number as the $n=1$ special case. Of particular interest is the $n\to \infty$ limit. To define it, first define $$U(x)=1_{x\geq 0}-1_{x\geq a}-1_{x\geq b}+1_{x\geq a+b},$$ a sum of step functions. Let $\mathbb{Z}^G$ be the set of integer-valued functions on $G$. Define a quadratic form on $\mathbb{Z}^G$ by 
\begin{equation}
    \label{def:quadratic-form}
    Q(\mathbf{n})=\sum_{i,j\in G} U(j-i) \,n_i n_j.
\end{equation}
Recall $(q)_m:=(q;q)_m$ is the $q$-Pochhammer symbol, and we adopt the convention $1/(q)_m=0$ for $m\in \mathbb{Z}_{<0}$. 

Now define
\begin{equation}\label{eq:inf-sum-def}
    Z_{a,b}(q):=\sum_{\mathbf{n}\in \mathbb{Z}^G} q^{Q(\mathbf{n})} \prod_{i\in G} \frac{(q)_{n_i-n_{i-a-b}}}{(q)_{n_i-n_{i-a}}(q)_{n_i-n_{i-b}}},
\end{equation}
where by convention $n_i=0$ for $i<0$. Due to the $q$-Pochhammers in the denominator, the summation vanishes outside the cone
\begin{equation}
    C_\mathbb{R}=\set*{\mathbf{n}\in \R^G: n_i\geq 0, n_i\geq n_j \text{ if }i-j\in \Gamma}.
\end{equation}
Notably, $Z_{a,b}(q)$ converges for $\abs{q}<1$ because it is proved in \cite{Huang2026} that $Q$ is positive definite on $C_\mathbb{R}$; in fact,
\[ Q(\mathbf{n})\geq \frac{1}{|G|^2}\lVert \mathbf{n}\rVert_2^2 \text{ for }\mathbf{n}\in C_\mathbb{R},\]
so only finitely many terms contribute to any given coefficient of $Z_{a,b}(q)$.

The infinite-rank $a\times b$ $q$-Catalan series $Z_{a,b}(q)$ arises from algebraic geometry as follows.

\begin{theorem}
    [Huang--Jiang--Oblomkov \cite{HJO}]
    Let $a,b>1$ be coprime and assume the notation above. For a prime power $q$, consider the variety of nilpotent matrix points on the algebraic curve $x^a=y^b$ whose set of $\Fq$-points is defined as
    \[ M_n(\Fq):=\set*{(A,B):A,B\in \mathrm{Nilp}_n(\Fq), AB=BA, A^a=B^b},\]
    and consider the infinite sum of nonnegative real numbers
    \[ S_q=\sum_{n=0}^\infty \frac{\abs*{M_n(\Fq)}}{\abs{\mathrm{GL}_n(\Fq)}} = \sum_{n=0}^\infty \frac{\abs*{M_n(\Fq)}}{(q^n-1)(q^n-q)\cdots (q^n-q^{n-1})}.\]

    Then
    \[ S_q=Z_{a,b}(q^{-1})\prod_{n=1}^\infty (1-q^{-n})^{-1}<\infty.\]
 \label{thm:HJO-conditional}
\end{theorem}

\medskip

It is natural to ask what the series $Z_{a,b}(q)$ evaluates to. HJO conjectured that $Z_{a,b}(q)$ is an infinite product, giving a striking bi-infinite family of Rogers--Ramanujan type identities.

To state the conjecture, we begin by defining the infinite product. For coprime $a,b> 1$, define the \defn{$a,b$-charge} at $i\in \Z$ as
$$r_{a,b}(i):=\min\{a,b,\mathrm{dist}(ai,(a+b)\Z)\}-1+1_{(a+b)\Z}(i),$$
where $\mathrm{dist}(ai,(a+b)\Z)$ is the distance from $ai$ to the closest multiple of $a+b$. Define
\[ P_{a,b}(q):=\prod_{i\geq 1}(1-q^i)^{-r_{a,b}(i)},\]
which converges for $\abs{q}<1$. It is worth noting that $r_{a,b}(i)=r_{a,b}(a+b\pm i)\geq 0$ and $\sum_{i=1}^{a+b}r_{a,b}(i)=(a-1)(b-1)$. Thus, $P_{a,b}(q)^{-1}$ is a product of $(a-1)(b-1)$ modular units of the form $\vartheta_{i;a+b}(q)=(q^i,q^{a+b-i};q^{a+b})_\infty$.



\begin{conjecture}
    [The Huang-Jiang-Oblomkov conjecture \cite{HJO}]
     Let $a,b> 1$ be coprime and assume the notation above. Then for $\abs{q}<1$, $Z_{a,b}(q)$ converges and
     \[ Z_{a,b}(q)=P_{a,b}(q).\] \label{conj:HJO}
\end{conjecture}
For $a=2$ and $b=2k+1$, the conjecture states that
\[ \sum_{r_1\geq \dots\geq r_k\geq 0} \frac{q^{r_1^2+\dots+r_k^2}}{(q)_{r_1-r_2}\dots (q)_{r_{k-1}-r_k}(q)_{r_k}} = \frac{(q^{k+1},q^{k+2},q^{2k+3};q^{2k+3})_\infty}{(q)_\infty},\]
where we have renamed the variables in $Z_{a,b}(q)$ by $r_i=n_{2k+1-2i}$. By the Andrews--Gordon identity, the HJO conjecture is known to hold for $a=2$. For all other cases, the HJO conjecture was open.

Our main result proves the next four cases of the HJO conjecture.

\begin{theorem}\label{thm:main}
    The HJO conjecture holds for $a=3$ and $b=4,5,7,8$.
\end{theorem}

\begin{remark}\label{rmk:lauono}
After the present work was completed, \lau and \ono \cite{lauono}
proved the HJO conjecture for $a=3$ and \emph{every} $b>3$ coprime
to $3$, by a different argument.
Their proof proceeds by verifying Conjecture~\ref{conj:sum}. Theorem~\ref{thm:main} is therefore subsumed by their
result. We retain the proofs presented here because they are of
independent interest: those for $b=4,5,7$ are short and self-contained,
using nothing beyond the $q$-Vandermonde and $q$-absorption identities,
while the proof for $b=8$ illustrates a $q$-holonomic method that
applies precisely when no such elementary route is available (see
Remark~\ref{rmk:fail}).
\end{remark}

Our proof naturally motivates a new family of conjectured sum-to-sum identities below.
\begin{conjecture}[Sum-to-sum conjecture]\label{conj:sum}
For $b>3$ coprime to $a=3$, let $f=\max G=2b-3$, and $k$ be the unique integer closest to $b/3$. Let $Q_{3,b}(\mathbf{n})$ be the quadratic form \eqref{def:quadratic-form}. For any fixed $r_1 \geq  r_2 \geq \dots \geq r_k \geq 0$, we have the following, where we define $n_i=0$ for $i<0$.
\begin{itemize}
    \item If $b=3k+1$, then 
    \begin{equation}\label{eq:sum-3k+1}
        \begin{multlined}
            \sum_{\mathbf{n}} q^{Q_{3,b}(\mathbf{n})}
            \left( \prod_{j=1}^{k} \qbinom{n_{3j+2}}{n_{3j-1}} \right)\prod_{j=1}^{k} \qbinom{r_{k-j+1}-n_{3j-5}}{n_{3j-2}-n_{3j-5}} \\
            = \sum_{m_1,\dots,m_k} q^{\sum_{i=1}^k (r_i^2 - r_i m_i + m_i^2)}
            \qbinom{2r_k}{m_k} \prod_{i=1}^{k-1} \qbinom{r_i - r_{i+1} + m_{i+1}}{m_i},
        \end{multlined}
    \end{equation}
    where the sum is over $\mathbf{n}\in \Z^G$ subject to $n_{f+3-3j}=r_j$ for $j=1,\dots,k$.

    \item If $b=3k-1$, then
    \begin{equation}\label{eq:sum-3k-1}
        \begin{multlined}
            \sum_{\mathbf{n}} q^{Q_{3,b}(\mathbf{n})}
            \left( \prod_{j=1}^{k-1} \qbinom{n_{3j+1}}{n_{3j-2}} \right)\prod_{j=1}^{k-1} \qbinom{r_{k-j}-n_{3j-4}}{n_{3j-1}-n_{3j-4}} \\
            = \sum_{m_1,\dots,m_{k-1}} q^{r_k^2+\sum_{i=1}^{k-1} (r_i^2 - r_i m_i + m_i^2)}
            \qbinom{r_{k-1} + r_k}{m_{k-1}} \prod_{i=1}^{k-2} \qbinom{r_i - r_{i+1} + m_{i+1}}{m_i},
        \end{multlined}
    \end{equation}
    where the sum is over $\mathbf{n}\in \Z^G$ subject to $n_{f+3-3j}=r_j$ for $j=1,\dots,k$.
\end{itemize}
\end{conjecture}
By identities of Warnaar \cite[(1.10),(1.11)]{WarnaarAG}, a straightforward argument (see Section~\ref{sec:sum-implies-hjo}) shows that Conjecture~\ref{conj:sum} for $b$ implies Conjecture~\ref{conj:HJO} for $(3,b)$. In fact, Conjecture~\ref{conj:sum} is finer than the $(3,b)$ case of Conjecture~\ref{conj:HJO}; the latter is proved by summing the former over $r_1,\dots,r_k$. 

\begin{theorem}\label{thm:sum}
    Conjecture~\ref{conj:sum} holds for $b=4,5,7,8$.
\end{theorem}

Theorem~\ref{thm:main} follows at once: combine Theorem~\ref{thm:sum}
with the implication recorded above, which is proved in
Section~\ref{sec:sum-implies-hjo}.

For general $b$, we are able to prove the $q=1$ specialization.
\begin{theorem}\label{thm:q=1}
    Conjecture~\ref{conj:sum} holds if $q=1$.
\end{theorem}

Although Theorem~\ref{thm:q=1} is subsumed by the result of \cite{lauono} recorded in Remark~\ref{rmk:lauono}, we include its proof because it is a direct bijection between flags of nested sets and uses no $q$-series machinery at all.

\section*{Acknowledgements} The authors thank George Andrews and Seewoo Lee for inspiring and helpful discussion regarding the contents of this paper.

\section{Notation and nuts and bolts}
\subsection*{$q$-Pochhammer symbols}
Throughout we assume $|q|<1$ so that infinite products of the form $(a;q)_\infty$ converge.
For $n\in\mathbb{Z}_{\ge 0}$ and complex $a$, define the \emph{$q$-Pochhammer symbol}
\[
\qPochN{a}{q}{n}:=\prod_{j=0}^{n-1}\bigl(1-aq^j\bigr),
\qquad
\qPochN{a}{q}{0}:=1,
\]
and the infinite product
\[
\qPochN{a}{q}{\infty}:=\prod_{j=0}^{\infty}\bigl(1-aq^j\bigr).
\]
We also use the common shorthand
\[
(q)_n := \qPochN{q}{q}{n}
\qquad {\textrm{and}} \qquad
(q)_\infty := \qPochN{q}{q}{\infty}.
\]
Standard references for this notation and basic identities include \cite[Chapter~1]{GasperRahman}.

\subsection*{Gaussian $q$-binomial coefficients}
For integers $N\ge 0$ and $k\in\mathbb{Z}$, define the \emph{$q$-binomial coefficient} by
\[
\qbinom{N}{k}
:=
\begin{cases}
\dfrac{(q)_N}{(q)_k\,(q)_{N-k}}, & 0\le k\le N,\\[1.0em]
0, & \text{otherwise}.
\end{cases}
\]
(See e.g.\ \cite[\S1.3]{GasperRahman}.) In particular, for $0\le k\le N$,
\[
\frac{1}{(q)_k\,(q)_{N-k}} = \frac{\qbinom{N}{k}}{(q)_N}.
\]

\section{The $q$-Vandermonde Identity}
The following multi-part $q$-Vandermonde identity is standard.  
\begin{lemma}[$k$-part $q$-Vandermonde Identity]\label{lem:q-vander-k}
Let $N = \sum_{i=1}^k N_i$. Then
\begin{equation}\label{eq:q-vander-k}
    \qbinom{N}{R} = \sum_{r_1 + \dots + r_k = R} q^{\sum_{1\le i < j \le k} (N_i - r_i)r_j} \prod_{i=1}^k \qbinom{N_i}{r_i}.
\end{equation}
\end{lemma}
\begin{remark}
    \label{rmk:ordering}
    By permuting the pairs $(N_i,r_i)$, Lemma~\ref{lem:q-vander-k} produces $k!$ identities with distinct powers of $q$. As such, to specify how one applies Lemma~\ref{lem:q-vander-k}, we say \eqref{eq:q-vander-k} is the $q$-Vandermonde Identity with \defn{block ordering} $(N_1,r_1),\dots,(N_k,r_k)$. For example, the $q$-Vandermonde Identity with block ordering $(A,k),(B,r-k)$ is
    \begin{equation}\label{eq:q-vander-2}
    \qbinom{A+B}{r} = \sum_{k=0}^r q^{(A-k)(r-k)} \qbinom{A}{k} \qbinom{B}{r-k},
    \end{equation}
    while the $q$-Vandermonde Identity with block ordering $(B,r-k),(A,k)$ is
    \begin{equation}
    \qbinom{A+B}{r} = \sum_{k=0}^r q^{(B-r+k)k} \qbinom{A}{k} \qbinom{B}{r-k}.
\end{equation}
\end{remark}

\section{The $q=1$ specialization and the proof of Theorem~\ref{thm:q=1}}

The proof of Theorem~\ref{thm:q=1} is combinatorial: both sides, evaluated at $q=1$, count the same set of ``flags of nested sets''. We abstract the key combinatorics in the following lemma to unify the proofs of the case $b=3k+1$ and the case $b=3k-1$. 

\begin{lemma}\label{lem:flag-straightening}
Let $\ell\geq 1$, let $r_1\geq\cdots\geq r_\ell\geq 0$, and let $s\geq 0$.
Put $x_0=0$ and $y_{\ell+1}=s$.  Then
\begin{equation}\label{eq:flag-straightening}
\begin{aligned}
&\sum_{\mathbf{x},\mathbf{y}\in\mathbb Z^{\ell}}
 \prod_{j=1}^{\ell}
 \binom{r_{\ell+1-j}-x_{j-1}}{x_j-x_{j-1}}
 \binom{y_{j+1}}{y_j} \\
&\hspace{35mm}=
\sum_{\mathbf{m}\in\mathbb Z^{\ell}}
 \binom{r_\ell+s}{m_\ell}
 \prod_{i=1}^{\ell-1}
 \binom{r_i-r_{i+1}+m_{i+1}}{m_i}.
\end{aligned}
\end{equation}
\end{lemma}

\begin{proof}
Fix nested sets
\[
 R_\ell\subseteq R_{\ell-1}\subseteq\cdots\subseteq R_1,
 \qquad |R_i|=r_i,
\]
and a further set $T$, disjoint from $R_1$, of size $s$.  The left-hand side
counts tuples of sets $(X_1,\dots,X_\ell,Y_1,\dots,Y_\ell)$ such that
\begin{equation}\label{eq:two-flags}
 X_1\subseteq\cdots\subseteq X_\ell,
 \qquad
 Y_1\subseteq\cdots\subseteq Y_\ell\subseteq T, \qquad X_j\subseteq R_{\ell+1-j} \text{ for } 1\le j\le \ell;
\end{equation}
the summation variables keep track of $|X_j|=x_j$ and $|Y_j|=y_j$. Indeed, we choose the sets in the order: $Y_\ell, \dots, Y_2, Y_1, X_1, X_2,\dots, X_\ell$; after $X_{j-1}$ has
been chosen, the only constraint on the choice of $X_j$ is $X_{j-1}\subseteq X_j\subseteq R_{\ell+1-j}$.

We give a bijection from these tuples of sets to the objects counted by the
right-hand side.  For $1\leq i\leq \ell$, set
\[
 D_i=R_i\setminus R_{i+1}, \qquad R_{\ell+1}\coloneqq \varnothing.
\]
Then the right-hand side of \eqref{eq:flag-straightening} counts tuples of sets $(M_1,\dots,M_\ell)$ such that
\begin{equation}\label{eq:recursive-subsets}
 M_p\subseteq  M_{p+1} \sqcup D_p\text{ for }1\leq p<\ell, M_\ell\subseteq T\sqcup D_\ell.
\end{equation}

To establish the bijection, given $(X_i, Y_i)$, define
\begin{equation}\label{eq:antidiagonal-sets}
 M_p=Y_p\sqcup \bigsqcup_{i=p}^\ell (X_{\ell+p-i}\cap D_i) \subseteq Y_{p+1}\sqcup \bigsqcup_{i=p}^\ell D_i,\qquad 1\leq p\leq \ell,
\end{equation}
where $Y_{\ell+1}\coloneqq T$.

The required inclusions \eqref{eq:recursive-subsets} for $1\leq p<\ell$ follows from the inclusions $Y_p\subseteq Y_{p+1}$ and $X_j\cap D_i\subseteq X_{j+1}\cap D_i$.

Conversely, given $(M_i)_{i=1}^\ell$, since $T\sqcup R_1=T\sqcup \bigsqcup_{i=1}^\ell D_i$ is a disjoint union, we recover
\[ Y_p=M_p\cap T, \qquad X_{\ell+p-i}\cap D_i=M_p\cap D_i,\]
and the latter can be rewritten as
\[ X_j \cap D_i=M_{i+j-\ell}\cap D_i, \qquad \ell+1-j\leq i\leq \ell.\]
Since $X_j\subseteq R_{\ell+1-j}$, $X_j$ is disjoint from $R_1,\dots,D_{\ell-j}$. The above thus determines $X_j$ by
\[ X_j=\bigsqcup_{i=\ell+1-j}^\ell M_{i+j-\ell}\cap D_i.\]
This completes the construction of the inverse map. Indeed, to verify $X_j\subseteq X_{j+1}$ for $1\leq j<\ell$, note by \eqref{eq:recursive-subsets} that
\[ M_p\cap D_i \subseteq M_{p+1}\cap D_i\]
if $i\neq p$, so for $\ell+1-j\leq i\leq \ell$, 
\[ M_{i+j-\ell}\cap D_i \subseteq M_{i+j-\ell+1}\cap D_i\]
because $i+j-\ell<i$. 
\end{proof}

\begin{proof}[Proof of Theorem~\ref{thm:q=1}]
At $q=1$, the powers of $q$ disappear and the Gaussian binomial
coefficients become ordinary binomial coefficients.

Suppose first that $b=3k+1$.  Apply Lemma~\ref{lem:flag-straightening} with
\[
 \ell=k,\qquad s=r_k,\qquad
 x_j=n_{3j-2},\qquad y_j=n_{3j-1}.
\]
Here $x_0=n_{-2}=0$ and $y_{k+1}=n_{3k+2}=r_k$.  Then \eqref{eq:flag-straightening} exactly gives the $q=1$ specialization of \eqref{eq:sum-3k+1}.

If $b=3k-1$, apply the same lemma with
\[
 \ell=k-1,\qquad s=r_k,\qquad
 x_j=n_{3j-1},\qquad y_j=n_{3j-2}.
\]
Now $x_0=n_{-1}=0$ and $y_k=n_{3k-2}=r_k$.  This gives
\eqref{eq:sum-3k-1} at $q=1$, finishing the proofs of both cases.
\end{proof}

\begin{remark}\label{rmk:fail}
The preceding bijection can routinely be written as an algebraic proof with the ordinary Vandermonde identity and the absorption identity
\[
 \binom{N}{A}\binom{A}{B}
 =\binom{N}{B}\binom{N-B}{A-B}.
\]
A $q$-analogue of this proof, using the $q$-Vandermonde identity (Lemma~\ref{lem:q-vander-k}) with suitable block orderings and the $q$-absorption identity
\[
 \qbinom{N}{A}\qbinom{A}{B}
 =\qbinom{N}{B}\qbinom{N-B}{A-B},
\]
yields the proofs for $b=4,5,7$ below. However, this tactic does not work in general: application of $q$-Vandermonde changes the power of $q$ depending on the block ordering, and already for $b=8$ and $b=10$ no ordering produces the correct power required by $Q_{3,b}$.
\end{remark}

\section{Proof of Theorem~\ref{thm:sum} for $b=4,5,7$}
\subsection{Proof of $b=4$}
The $b=4$ special case of Conjecture~\ref{conj:sum} reads:
\begin{lemma}\label{lem:sum-reduction-34}
    For any $n_5=r_1\geq 0$, we have
    \begin{equation}
        \sum_{n_1,n_2} q^{Q(\mathbf{n})} \qbinom{r_1}{n_1}\qbinom{r_1}{n_2} 
        = \sum_{m_1} q^{r_1^2-r_1 m_1+m_1^2}\qbinom{2r_1}{m_1},
    \end{equation}
    where
    \begin{equation*}
        Q(\mathbf{n})=n_1^2+n_2^2+n_5^2+n_1n_2-n_1n_5
    \end{equation*}
\end{lemma}

\begin{proof}
Let $S(q)$ denote the left-hand side, with fixed parameter $r_1 = n_5$. The sum is taken over $\mathbf{n} = (n_1, n_2)$:
\begin{equation*}
    S(q) = \sum_{\mathbf{n}} q^{Q(\mathbf{n})} \qbinom{r_1}{n_1} \qbinom{r_1}{n_2}.
\end{equation*}

We evaluate the summation directly by applying the $q$-Vandermonde identity (Lemma \ref{lem:q-vander-k}). 

We consider the sum over $n_1, n_2$, subject to the constraint $n_1 + n_2 = m_1$. By the $q$-Vandermonde identity with the block ordering $(r_1, n_1), (r_1, n_2)$ (see Remark~\ref{rmk:ordering}), we have
\begin{equation}
    \sum_{n_1+n_2=m_1} q^{(r_1-n_1)n_2} \qbinom{r_1}{n_1} \qbinom{r_1}{n_2} = \qbinom{2r_1}{m_1}.
\end{equation}

Miraculously, the remaining quadratic form
\begin{equation*}
    Q_1 = Q(\mathbf{n}) - (r_1-n_1)n_2
\end{equation*}
depends on $n_1, n_2$ strictly through $m_1$. Indeed, 
\begin{equation*}
    Q_1 = r_1^2 - r_1 m_1 + m_1^2.
\end{equation*}

As a result, we may factor out $q^{Q_1}$ and complete the evaluation:
\begin{align*}
    S(q) &= \sum_{m_1} q^{Q_1} \sum_{\substack{n_1,n_2\\n_1+n_2=m_1}} q^{(r_1-n_1)n_2} \qbinom{r_1}{n_1} \qbinom{r_1}{n_2} \\
    &= \sum_{m_1} q^{r_1^2 - r_1 m_1 + m_1^2} \qbinom{2r_1}{m_1}.
\end{align*}
This precisely matches the desired right-hand side, completing the proof.
\end{proof}

\subsection{Proof of $b=5$}
The $b=5$ special case of Conjecture~\ref{conj:sum} reads:
\begin{lemma}\label{lem:sum-reduction-35}
    For any $n_7=r_1\geq n_4=r_2\geq 0$, we have
    \begin{equation}
        \sum_{n_1,n_2} q^{Q(\mathbf{n})} \qbinom{r_2}{n_1}\qbinom{r_1}{n_2} 
        = \sum_{m_1} q^{r_1^2-r_1 m_1+m_1^2+r_2^2}\qbinom{r_1+r_2}{m_1},
    \end{equation}
    where
    \[
    Q(\mathbf{n})=n_1^2+n_2^2+n_4^2+n_7^2+n_1n_2+n_2n_4-n_1n_7-n_2n_7.
    \]
\end{lemma}

\begin{proof}
Let $S(q)$ denote the left-hand side, with fixed parameters $r_1 = n_7$ and $r_2 = n_4$. The sum is taken over $\mathbf{n} = (n_1, n_2)$:
\begin{equation*}
    S(q) = \sum_{\mathbf{n}} q^{Q(\mathbf{n})} \qbinom{r_2}{n_1} \qbinom{r_1}{n_2}.
\end{equation*}

We evaluate the summation directly by applying the $q$-Vandermonde identity (Lemma \ref{lem:q-vander-k}). 

We consider the sum over $n_1, n_2$, subject to the constraint $n_1 + n_2 = m_1$. By the $q$-Vandermonde identity with the block ordering $(r_2, n_1), (r_1, n_2)$ (see Remark~\ref{rmk:ordering}), we have
\begin{equation}
    \sum_{n_1+n_2=m_1} q^{(r_2-n_1)n_2} \qbinom{r_2}{n_1} \qbinom{r_1}{n_2} = \qbinom{r_1+r_2}{m_1}.
\end{equation}

Miraculously, the remaining quadratic form
\begin{equation*}
    Q_1 = Q(\mathbf{n}) - (r_2-n_1)n_2
\end{equation*}
depends on $n_1, n_2$ strictly through $m_1$. Indeed, 
\begin{equation*}
    Q_1 = r_1^2 - r_1 m_1 + m_1^2+r_2^2.
\end{equation*}

As a result, we may factor out $q^{Q_1}$ and complete the evaluation:
\begin{align*}
    S(q) &= \sum_{m_1} q^{Q_1} \sum_{\substack{n_1,n_2\\n_1+n_2=m_1}} q^{(r_2-n_1)n_2} \qbinom{r_2}{n_1} \qbinom{r_1}{n_2} \\
    &= \sum_{m_1} q^{r_1^2 - r_1 m_1 + m_1^2+r_2^2} \qbinom{r_1+r_2}{m_1}.
\end{align*}
This precisely matches the desired right-hand side, completing the proof.
\end{proof}

\subsection{Proof of $b=7$}
The $b=7$ special case of Conjecture~\ref{conj:sum}  reads:
\begin{lemma}\label{lem:sum-reduction-37}
    For any $n_{11}=r_1\geq n_8=r_2\geq 0$, we have
    \begin{equation}
        \begin{multlined}
            \sum_{n_1,n_2,n_4,n_5}
q^{Q(\mathbf{n})}
\qbinom{n_8}{n_5}\qbinom{n_5}{n_2}\qbinom{n_8}{n_1}\qbinom{n_{11}-n_1}{n_4-n_1} 
\\= \sum_{m_1,m_2} q^{r_1^2-r_1 m_1+m_1^2+r_2^2-r_2m_2+m_2^2}\qbinom{r_1-r_2+m_2}{m_1}\qbinom{2r_2}{m_2},
        \end{multlined}
    \end{equation}
    where
    \begin{equation}
    Q(\mathbf{n}) = 
    \left(\sum_{i\in \{1,2,4,5,8,11\}} n_i^2\right) + n_1n_2 + n_2n_4 + n_4n_5 - n_1n_8 - n_2n_{11} - n_4n_{11},
    \end{equation}
\end{lemma}

\begin{proof}
Let $S(q)$ denote the left-hand side.
Applying Lemma~\ref{lem:q-vander-k} to the last $q$-binomial coefficient with block ordering
\(
    (r_1-r_2,c), (r_2-n_1,a-n_1),
\)
and using
\[
    \qbinom{r_2}{n_1}\qbinom{r_2-n_1}{a-n_1}
    =\qbinom{r_2}{a}\qbinom{a}{n_1},
\]
we obtain
\[
    S(q)=\sum_{n_1,n_2,a,c,n_5}q^{Q_1}
    \qbinom{a}{n_1}\qbinom{n_5}{n_2}
    \qbinom{r_1-r_2}{c}
    \qbinom{r_2}{a}\qbinom{r_2}{n_5},
\]
where
\[
    Q_1=
    \left.Q(\mathbf n)\right|_{n_4=a+c}
    +(r_1-r_2-c)(a-n_1).
\]

Set
\[
    m_1=n_1+n_2+c,\qquad m_2=a+n_5.
\]
The $q$-Vandermonde identity with block ordering
\[
    (a,n_1),\qquad(n_5,n_2),\qquad(r_1-r_2,c)
\]
gives
\[
    \sum_{n_1+n_2+c=m_1}q^{E_1}
    \qbinom{a}{n_1}\qbinom{n_5}{n_2}
    \qbinom{r_1-r_2}{c}
    =\qbinom{r_1-r_2+m_2}{m_1},
\]
where
\[
    E_1=(a-n_1)n_2+(a-n_1)c+(n_5-n_2)c.
\]

A direct expansion gives
\[
    Q_2\coloneqq Q_1-E_1=r_1^2-r_1 m_1+m_1^2
    +a^2+an_5+n_5^2-ar_2+r_2^2,
\]
which does not depend on $n_1,n_2,c$.
Consequently,
\begin{equation}\label{eq:b=7}
    S(q)=\sum_{m_1,a,n_5}
    q^{Q_2}
    \qbinom{r_1-r_2+m_2}{m_1}
    \qbinom{r_2}{a}\qbinom{r_2}{n_5}.
\end{equation}

Now put $m_2=a+n_5$. The $q$-Vandermonde identity with block ordering
$ (r_2,a),(r_2,n_5)$ yields
\[
    \sum_{a+n_5=m_2}q^{E_2}
    \qbinom{r_2}{a}\qbinom{r_2}{n_5}
    =\qbinom{2r_2}{m_2},
\]
where $E_2=(r_2-a)n_5$.  Whenever $a+n_5=m_2$, we have
\[
    Q_3\coloneqq Q_2-E_2
    =r_1^2-r_1 m_1+m_1^2+r_2^2-r_2 m_2+m_2^2,
\]
which again does not depend on $a,n_5$. 

Grouping the right-hand side of \eqref{eq:b=7} by $m_2$ therefore gives
\[
    S(q)=\sum_{m_1,m_2}
    q^{Q_3}
    \qbinom{r_1-r_2+m_2}{m_1}
    \qbinom{2r_2}{m_2},
\]
which is the desired identity.
\end{proof}

\section{Proof of Theorem~\ref{thm:sum} for $b=8$}
The proof method in $b=4,5,7$ no longer works; see Remark~\ref{rmk:fail}. Instead, we use the annihilator method.

\subsection{The setting}\label{subsec:setting}

Define quadratic forms
\begin{align*}
Q(r_1,r_2,r_3,n_1,n_2,n_4, n_5)&:=
r_1^2+r_2^2+r_3^2+n_1^2+n_2^2+n_4^2+n_5^2\\
&\hspace*{0.5cm}
+  n_1 n_2 + n_2 n_4 + n_4 n_5 + n_5 r_3 
- n_1 r_2 - n_2 r_2 - n_4 r_1 - n_5 r_1
\end{align*}
and
\[
P(r_1,r_2,r_3,m_1,m_2):=
r_1^2 - r_1 m_1 + m_1^2 + r_2^2 - r_2 m_2 
+ m_2^2 + r_3^2.
\]
Define
\begin{eqnarray}\label{eq:S4 sd}
s_4(r_1,r_2,r_3,n_1,n_2,n_4, n_5):=
q^{Q(r_1,r_2,r_3,n_1,n_2,n_4, n_5)}
\qbin{r_3}{n_4}{q}\qbin{n_4}{n_1}{q}
\qbin{r_2}{n_2}{q}\qbin{r_1-n_2}{n_5-n_2}{q}
\end{eqnarray}
and
\begin{eqnarray}\label{eq:S2 sd}
s_2(r_1,r_2,r_3,m_1,m_2):=
q^{P(r_1,r_2,r_3,m_1,m_2)}
\qbin{r_1-r_2+m_2}{m_1}{q}\qbin{r_2+r_3}{m_2}{q}.
\end{eqnarray}

Then the $b=8$ case of Conjecture~\ref{conj:sum} is equivalent to
\begin{lemma}\label{lem:4}
Let
\begin{equation}\label{eq:S4}
S_4(r_1,r_2,r_3):=
\sum_{n_1,n_2,n_4,n_5\geq 0} 
s_4(r_1,r_2,r_3,n_1,n_2,n_4,n_5)
\end{equation}
and
\begin{equation}\label{eq:S2}
S_2(r_1,r_2,r_3):=
\sum_{m_1,m_2\geq 0} 
s_2(r_1,r_2,r_3,m_1,m_2).
\end{equation}
Then for $r_1\geq r_2 \geq r_3\geq 0$,
\begin{equation}
S_4(r_1,r_2,r_3)=S_2(r_1,r_2,r_3).
\end{equation}
\end{lemma}


Our proof of Lemma~\ref{lem:4} is heavily
computer-assisted. It makes substantial
use of computer algebra packages developed
at RISC which are freely available
at \url{https://combinatorics.risc.jku.at/software}.

All Mathematica computations have been
executed on a standard laptop: Lenovo,
32 GB memory, Intel Core i7-1260P. The
run times of the procedure calls range
from fractions of a second to roughly
one minute.

Despite relying on sophisticated algorithms,
a general important aspect of our computational
proof methodology is this:
owing to the structure of
the algorithms used, all our computations
can be verified by elementary means 
(essentially high-school algebra)
\textit{independently from the algorithmic
steps which computed them}; see 
Subsection~\ref{subsec:Prop1ProofIndep}.
\subsection{Computer Algebra Packages used}\label{sec:CA}

\subsubsection{\texttt{qMultiSum} to compute
recurrences}\label{subsec:qMult}

The main package used is Axel Riese's Mathematica implementation~\cite{Riese} of $q$-WZ theory.

\begin{mma}
\In << RISC`qMultiSum`\\
\Print \LoadP{Version 2.54 written by Axel Riese
\copyright\ RISC, Johannes Kepler University Linz}\\
\end{mma}
We will use this package to derive a variety
of recurrences for the sums $S_j(r_1, r_2,r_3)$, $j\in\{2,4\}$. To this end, we need to input
the respective summands:
\begin{mma}
\In
P[m1\_, m2\_, r3\_, r2\_, r1\_]:=
r1^2 - r1\, m1 + m1^2 + r2^2 - r2\, m2 
+ m2^2 + r3^2
\\
\In
S2Summand[r1\_, r2\_, r3\_, m1\_, m2\_] := 
 q^{P[m1, m2, r3, r2, r1]}
 qBinomial[r1 - r2 + m2, m1, q] 
 \newline \hspace*{6.1cm}\ast qBinomial[r2 + r3, m2, q]
\\
\end{mma}
and
\begin{mma}
\In
Q[n1\_, n2\_, n4\_, n5\_, r3\_, r2\_, r1\_]:=
r1^2+r2^2+r3^2+n1^2+n2^2+n4^2+n5^2
+  n1\, n2 + n2\, n4 
\newline \hspace*{5.5cm}
+ n4\, n5 + n5\, r3 
- n1\, r2 - n2\, r2 - n4\, r1 - n5\, r1
\\
\In
S4Summand[r1\_, r2\_, r3\_, n1\_, n2\_, n4\_, n5\_] := 
 q^{Q[n1, n2, n4, n5, r3, r2, r1]} qBinomial[r3, n4, q]
  \newline \hspace*{7.2cm}
 \ast qBinomial[n4, n1, q] qBinomial[r2, n2, q] 
\newline \hspace*{7.2cm} 
 \ast qBinomial[r1 - n2, n5 - n2, q]
\\
\end{mma}

In addition to the basic $q$-WZ algorithm to compute
recurrences as outlined
in~\cite{WZ}, 
Riese's \texttt{qMultiSum} package 
implements important additional features
and extensions, for
example, ``Verbaeten completion''. 
Verbaeten's theory~\cite{Verbaeten}, which was
revisited and algorithmically exploited by Hornegger~\cite{Hornegger} and
Wegschaider~\cite{Wegschaider},
helps to input structure sets in an optimized 
fashion. A structure set is a subset of an integer lattice $\Z^n$ such that 
its points represent the shifts
arising in a recurrence. As an example,
we compute a recurrence for $S_2(r_1,r_2,r_3)$
using a \texttt{qMultiSum} procedure call,
\begin{mma}
\In
recs1 = 
\newline \hspace*{0.3cm}
qFindRecurrence[
   S2Summand[r1, r2, r3, m1, m2], \{r1, r2, r3\}, \{m1, m2\}, \{0, 0, 
    0\}, \{1, 1\}]; \label{mma:recs1}
\\
\In
Length[recs1]
\\
\Out
24
\\
\end{mma}
This means, the output \texttt{recs1} consists
of a list of $24$ entries, each of them being
a summand recurrence. Let us pick the entry at
position $2$,
\begin{mma}
\In
P1recSd=recs1[[2]]
\\
\Out
 -q^{2 r2 + r3} F[-2 + r1, -3 + r2, -1 + r3, m1, -1 + m2] - 
 q^{3 + 2 r3} F[-2 + r1, -2 + r2, -2 + r3, m1, m2]
\newline \hspace*{0.3cm}
+q^6 F[-2 + r1, -2 + r2, -1 + r3, m1, m2] = 0 
\label{mma:P1 SdRec}\\
\end{mma}
The expression \texttt{F[$\dots$]} is used
generically by the program; it has to
be interpreted as
\[
\texttt{F[r1,r2,r3,m1,m2]} =
s_2(r_1,r_2,r_3,m_1,m_2).
\]
Hence \texttt{Out[\ref{mma:P1 SdRec}]} is a recurrence for the summand of
$S_2(r_1,r_2,r_3)$.
Its structure set can be extracted from the shifts which are,
\begin{equation}\label{eq:VB1}
\{(2,3,1,0,1), (2,2,2,0,0), 
(2,2,1,0,0)\}\subseteq \Z^5.
\end{equation}

In the call \texttt{In[\ref{mma:recs1}]},
we specified a structure
set with ``\texttt{\{0,0,0\},\{1,1\}}.'' This tells the program to search
for summand recurrences whose structure sets 
are subsets of the corresponding Verbaeten completion of the set specified in this way;
here the shorthand notation refers to the set
\begin{equation*}
\{r_1,r_2,r_3,m_1,m_2): r_1=0, r_2=0, r_3=0,
0\leq m_1\leq 1, 0\leq m_2\leq 1\}.
\end{equation*}
We stress that Verbaeten completions are
heavily dependent on the particular form
of the given
summand expression; \texttt{qMultiSum} provides
an option to compute it,
\begin{mma}
\In
qFindRecurrence[
 S2Summand[r1, r2, r3, m1, m2], \{r1, r2, r3\}, 
 \{m1, m2\}, \{0, 0, 0\}, \{1, 1\},
 \newline \hspace*{5.3cm}
   OnlyStructSet \rightarrow True]
\\
\Out
\{
\{0, 0, 3, 0, 0\}, \{1, 0, 3, 0, 0\}, 
\{1, 1, 2, 0, 0\}, \{1, 1, 3, 0, 0\}, 
  \{1, 2, 2, 0, 1\}, \{2, 0, 3, 0, 0\}, 
  \{2, 1, 2, 0, 0\}, 
 \newline \hspace*{0.3cm}
  \{2, 1, 2, 1, 0\}, 
  \{2, 1, 3, 0, 0\}, \{2, 1, 3, 1, 0\}, 
  \{2, 2, 1, 0, 0\}, \{2, 2, 2, 0,  0\}, 
  \{2, 2, 2, 0, 1\}, \{2, 2, 2, 1, 1\}, 
\newline \hspace*{0.3cm} 
  \{2, 2, 3, 0, 0\}, \{2, 3, 1, 0, 1\}, 
 \{3, 1, 2, 0, 0\}, \{3, 1, 2, 1, 0\}, 
  \{3, 1, 3, 0, 0\}, \{3, 1, 3, 1, 
  0\}, 
  \{3, 2, 1, 0, 0\}, 
\newline \hspace*{0.3cm} 
  \{3, 2, 1, 1, 0\}, 
  \{3, 2, 2, 0, 0\}, \{3, 2, 2, 0, 
  1\}, 
  \{3, 2, 2, 1, 0\}, 
  \{3, 2, 2, 1, 1\}, \{3, 2, 3, 0, 0\}, 
  \{3, 2, 3, 1, 
  0\}, 
 \newline \hspace*{0.3cm}  
  \{3, 3, 1, 0, 1\}, \{3, 3, 1, 1, 1\}, 
  \{4, 2, 1, 1, 0\}, \{4, 2, 2, 1, 
  0\}, 
  \{4, 2, 3, 1, 0\}, 
  \{4, 3, 0, 1, 0\}, 
  \{4, 3, 1, 1, 0\}, 
 \newline \hspace*{0.3cm}  
  \{4, 3, 1, 1, 
  1\}, 
  \{4, 3, 2, 1, 0\}, \{4, 4, 0, 1, 1\}\}.
\label{mma:StrSet incl P1}
\\ 
\end{mma}
The size of this Verbaeten completion makes
it plausible,
at least in this case, why the program found $24$
summand recurrences where each of them
 has a  structure set being
a subset of \texttt{Out[\ref{mma:StrSet incl P1}]}; the subset~\eqref{eq:VB1} 
exemplifies this.

Next we continue our example where we
set the goal to compute a recurrence for the sum $S_2(r_1,r_2,r_3)$. To this end, notice that each of the
coefficients of
the summand recurrence \texttt{P1recSd} 
from~\texttt{Out[\ref{mma:P1 SdRec}]} is 
free of the summation variables $m_1$ and $m_2$.
Moreover, the summand 
$s_2(r_1,r_2,r_3,m_1,m_2) =
\texttt{F[r1,r2,r3,m1,m2]}$ is non-zero
for only finitely many points $(m_1,m_2)\in
\Z^2$; i.e., it has finite support with respect
to $m_1$ and $m_2$. As a consequence, this
allows us to flexibly sum 
both sides of the recurrence \texttt{P1recSd}
over all non-negative $m_1$ and $m_2$ to obtain
a recurrence for the sum $S_2(r_1,r_2,r_3)$.
This is automatized by a \texttt{MultipleSum}
command,
\begin{mma}
\In
P1rec=qSumRecurrence[P1recSd, 3]
\\
\Out
-q^{2 + 2 r2 + r3} SUM[r1, r2, 1 + r3] - 
  q^{1 + 2 r3} SUM[r1, 1 + r2, r3] + SUM[r1, 1 + r2, 1 + r3] = 0
\label{mma:P1rec}\\
\end{mma}
Again \texttt{SUM[$\dots$]} is used generically
by the program, and the output means, we obtained a recurrence for 
\[
S_2(r_1,r_2,r_3)=\texttt{SUM[r1,r2,r3]}.
\]
Notice that the sum recurrence is given
in normalized form; i.e., all shifts are
non-negative.

\begin{remark}
The ``$3$'' in ``\texttt{qSumRecurrence[P1recSd, 3]}''
tells the program to consider the first $3$ variables
in \texttt{F[r1,r2,r3,m1,m2]} as the independent variables;
in other words, this means,
summation is carried out with respect
to the remaining variables $m_1$ and $m_2$.
\end{remark}

\subsubsection{\texttt{HolonomicFunctions} to compute
relations between $q$-shift operators}\label{subsec:qHoloFus}

To successfully execute our strategy to prove
Lemma~\ref{lem:4}, we need to show that a
particular recurrence for $S_4(r_1,r_2,r_3)$,
computed with \texttt{qMultiSum},
is also a recurrence for $S_2(r_1,r_2,r_3)$.
To this end, we invoke  Christoph Koutschan's
RISC package \texttt{HolonomicFunctions}~\cite{Koutschan} written in Mathematica. 

\begin{mma}
\In << RISC`HolonomicFunctions`\\
\Print \LoadP{Version 1.7.3 (21-Mar-2017)
written by Christoph Koutschan
\newline
\copyright\ RISC, Johannes Kepler University Linz}\\
\end{mma}
This package implements various aspects and
extensions of Zeilberger's holonomic systems approach to special functions identities~\cite{Zholo}. For our proof
we will use a different feature of Koutschan's
package; namely, it implements non-commutative
Gr\"obner bases computations for Ore algebras.

To connect to our context, we transform 
recurrences obtained with \texttt{qMultiSum}
into $q$-shift operators. This transformation
is also supported by Koutschan's package.
To this end, we first
have to define the algebra \texttt{alg}
of $q$-shift operators,
\begin{mma}
\In
alg = OreAlgebra[QS[qr_1, q^{r_1}], QS[qr_2, q^{r_2}], QS[qr_3, q^{r_3}]]
\label{mma:alg}\\
\Out
\KK(q, qr_1, qr_2, qr_3)[S_{qr_1, q}; S_{qr_1, q}, 0]
[S_{qr_2, q}; S_{qr_2, q}, 0]
[S_{qr_3, q}; S_{qr_3, q}, 0]
\\
\end{mma}
This is an algebraic coding of the non-commutative
algebra of $q$-shift operators 
$S_{\text{qr}_1, q}$,
$S_{\text{qr}_2, q}$, and $S_{\text{qr}_3, q}$; the coefficient
domain of these operators are the rational 
functions $\KK(q, \text{qr}_1, \text{qr}_2, \text{qr}_3)$. On the level of the algebra
generators, non-commutativity of 
the $q$-shift operators
comes in as follows. For $j\in\{1,2,3\}$ ,
\begin{equation}\label{eq:non-comm}
S_{\text{qr}_j, q}\dast \text{qr}_j 
= 
q\, \text{qr}_j \dast S_{\text{qr}_j, q} ;
\end{equation}
i.e., the symbols $\text{qr}_j$ stand for the 
$q$-powers $q^{r_j}$. In our context, $\KK=\Q$.

Sometimes we need to make
non-commutativity explicit, in particular,
for certain inputs to 
\texttt{HolonomicFunctions}; in such cases
we shall write ``$\dast$'' for non-commutative
multiplication.  
Nevertheless, to keep things visually simple we
often omit to write ``$\dast$'' if things are clear
from the context; e.g., in case a power-product
of $q$-shifts stands to right of a term. This 
means, we shall write
\[
q\, \text{qr}_j S_{\text{qr}_j, q}
\text{\, instead of\, }
q\, \text{qr}_j \dast S_{\text{qr}_j, q}.
\]

We conclude this section with an example showing
how to transform the sum recurrence \texttt{P1rec} from \texttt{Out[\ref{mma:P1rec}]} into an element of this algebra of $q$-shifts,
\begin{mma}
\In
P_1 = ToOrePolynomial[
{QS}[{qr}_2,q^{r_2}]\dast
{QS}[{qr}_3,q^{r_3}]
+q^{2r_2+r_3+2}
   -{QS}[{qr}_3,q^{r_3}]
\newline \hspace*{4.0cm}   
   -q^{2 r_3+1}
   {QS}[{qr}_2,q^{r_2}], alg]
\\
\Out
S_{\text{qr}_2,q}S_{\text{qr}_3,q}-q
   \text{qr}_3^2S_{\text{qr}_2,q}-q^2 \text{qr}_2^2
   \text{qr}_3S_{\text{qr}_3,q}
\label{mma:P1}\\
\end{mma}
Translating back to recurrence form, this means,
{\small
\begin{align*}
P_1 S_2(r_1,r_2,r_3)
&= \left(
S_{\text{qr}_2,q}S_{\text{qr}_3,q}-q
   \text{qr}_3^2S_{\text{qr}_2,q}-q^2 \text{qr}_2^2
   \text{qr}_3S_{\text{qr}_3,q}
   \right) S_2(r_1,r_2,r_3)
   \\
&=S_2(r_1,r_2+1,r_3+1) - q^{1+2 r_3} S_2(r_1,r_2+1,r_3)- q^{2+2 r_2+r_3}
S_2(r_1,r_2,r_3+1)=0.
\end{align*}
}
As mentioned above, to show that a
particular recurrence for $S_4(r_1,r_2,r_3)$,
computed with \texttt{qMultiSum},
is also satisfied by $S_2(r_1,r_2,r_3)$,
in Subsection~\ref{subsec:Prop1Proof} we will transform recurrences into
operator form in order to invoke the 
non-commutative Gr\"obner basis procedure
from Koutschan's \texttt{HolonomicFunctions}
package. For further details on the functionalities of this software we refer
to~\cite{Koutschan}.

\subsection{Proof of Lemma~\ref{lem:4}}\label{sec:proof}

\subsubsection{Reduction $r_3+1\rightarrow r_3$}

The first step in our proof is to reduce
the task of proving
\[
S_4(r_1,r_2,r_3)=S_2(r_1,r_2,r_3),
\hspace*{0.3cm} r_1\geq r_2\geq r_3\geq 0,
\]
to proving
\begin{equation}\label{eq:R1R2}
S_4(R_1,R_2,0)=S_2(R_1,R_2,0),
\hspace*{0.3cm} R_1\geq R_2\geq 0.
\end{equation}

This reduction is accomplished by a recurrence which decreases the argument $r_3$ by one.
Such a recurrence is computed with \texttt{qMultiSum} by,
\begin{mma}
\In
Prec=
(qFindRecurrence[
  S2Summand[r1, r2, r3, m1, m2],  
  \{r1, r2, r3\}, \{m1, m2\},
 \newline  \hspace*{3.2cm}   
   \{3, 0, 0\}, \{0, 0\}, 1] 
   // qSumRecurrence[\#, 3] \&)[[1]]
\\
\Out
-q^{7 + 5 r1 + r2} SUM[r1, r2, 1 + r3] - 
  q^{6 + 3 r1 + 2 r3} SUM[1 + r1, 1 + r2, r3] 
  \newline \hspace*{0.3cm}
  - 
  q^{5 + 2 r1 + 2 r3} SUM[2 + r1, 1 + r2, r3] + 
  q^{2 r3} SUM[3 + r1, 1 + r2, r3] = 0
\label{mma:Prec}\\
\end{mma}
Inspection of the shifts shows, the 
recurrence \texttt{Prec} for
\[
S_2(r_1,r_2,r_3)=
\texttt{SUM[r1,r2,r3]}
\]
is valid for 
\begin{equation}\label{eq:ineqs}
 r_1 \geq r_2 \geq r_3+1 \geq 0.
\end{equation}
Consequently, computing
the values $S_2(r_1, r_2, r_3)$, 
$r_1\geq r_2 \geq r_3 \geq 1$, reduces
to finding the values $S_2(R_1, R_2, 0)$
for $R_1\geq R_2\geq 0$. 

In Subsection~\ref{subsec:Prop1Proof}
we shall prove the following fact.
\begin{prop}\label{prop:Prop1}
The recurrence \texttt{Prec} from 
\texttt{Out[\ref{mma:Prec}]} for the
double sum $S_2(r_1,r_2,r_3)$ is also
satisfied by the quadruple sum 
\[
S_4(r_1,r_2,r_3)=
\texttt{SUM[r1,r2,r3]},
\hspace*{0.3cm}
\text{where\, }
r_1 \geq r_2 \geq r_3+1 \geq 0.
\]
\end{prop}

\begin{remark}
We were not able to reproduce the recurrence 
\texttt{Prec} for $S_4(r_1, r_2, r_3)$ by an explicit \texttt{qMultiSum} procedure call as we did in \texttt{In[\ref{mma:Prec}]} for
$S_2(r_1, r_2, r_3)$.
This might be owing to the fact that there are (too) many choices for structure sets of
possible recurrences. In other words,
in practice it can be a complex task to find a particularly structured recurrence as a ``needle in a haystack.'' As a consequence, to 
prove Proposition~\ref{prop:Prop1} we will
use a different tool-box: Koutschan's
\texttt{HolonomicFunctions} for 
non-commutative Gr\"obner
bases computations.
\end{remark}

To proceed with our proof of Lemma~\ref{lem:4},
we now assume that Proposition~\ref{prop:Prop1}
is proven. To prove the equality~\eqref{eq:R1R2},
consider the \texttt{qMultiSum} procedure calls,
\begin{mma}
\In
P2rec=
(qFindRecurrence[
    S2Summand[r1, r2, r3, m1, m2], 
    \{r1, r2, r3\}, \{m1, m2\}, 
 \newline \hspace*{4.9cm}      
    \{0, 3, 0\}, \{0, 0\}, 
    \{1,1\}] 
    // qSumRecurrence[\#, 3] \&)[[8]]
\\
\Out
q^{3 + 3 r1 + 2 r2} SUM[r1, r2, r3] + 
  q^{2 + 2 r1} SUM[1 + r1, 1 + r2, r3] - 
  q^{2 r2} SUM[2 + r1, r2, r3] = 0
\label{mma:P2rec}\\
\end{mma}
and
\begin{mma}
\In
(qFindRecurrence[
    S4Summand[r1, r2, r3, n1, n2, n4, n5], 
    \{r1, r2, r3\}, \{n1, n2, n4, n5\}, 
  \newline \hspace*{3.5cm}       
    \{0, 0, 0\}, \{2, 1, 0, 1\},
     \{0, 0, 1, 1\}] // qSumRecurrence[\#, 3] \&)[[1]]
\\
\Out
q^{3 + 3 r1 + 2 r2} SUM[r1, r2, r3] + 
  q^{2 + 2 r1} SUM[1 + r1, 1 + r2, r3] - 
  q^{2 r2} SUM[2 + r1, r2, r3] = 0
\\
\end{mma}
This means, the recurrence \texttt{P2rec}
as in \texttt{Out[\ref{mma:P2rec}]} is
satisfied by both $S_2(r_1,r_2,r_3)$ and
$S_4(r_1,r_2,r_3)$.
After setting $r_3=0$ and replacing
the remaining
$r_j$ by $R_j$, it rewrites as,
\begin{equation}\label{eq:forCaseA zero}
q^{2 R_2} S_j(2 + R_1, R_2, 0)=
q^{2 + 2 R_1}
   S_j(1 + R_1, 1 + R_2, 0) + 
q^{3 + 3 R_1 + 2 R_2} S_j(R_1, R_2, 0),
\end{equation}
which is valid for $j\in\{2,4\}$ and
the domain
\begin{equation}\label{eq:domain}
R_1\geq R_2\geq 0.
\end{equation}

As a consequence, given the values of 
$S_j(R_1,R_2,0)$, $j\in\{2,4\}$, on the
boundary ``diagonals'',
\[
\{(R_1,R_1,0),R_1\geq 0\} 
\text{\, and\, }
\{(R_1,R_1-1,0),R_1\geq 1\},
\]
the remaining values in the
domain~\eqref{eq:domain} are
determined by the recurrence~\eqref{eq:forCaseA zero}.

To prove $S_2(R_1,R_1,0)=S_4(R_1,R_1,0)$ we
compute a common recurrence on the main diagonal
as follows,
\begin{mma}
\In
qFindRecurrence[
  S2Summand[R1, R1, 0, m1, m2], 
  \{r1\}, \{m1, m2\}, \{1\}, \{1, 1\},
\newline \hspace*{5.3cm}  
   \{1,1\}] // 
 qSumRecurrence[\#, 1] \&
\\
\Out
\{-q^{14 + 9 R1} (-1 + q^{1 + R1})
 (-1 + q^{2 + R1}) SUM[R1] - 
   q^{12 + 6 R1} (-1 + q^{2 + R1})
    (1 + q^{2 + R1}) SUM[1 + R1] 
\newline \hspace*{0.5cm}    
  + q^{8 + 3 R1} (1 + 2 q^{2 + R1}) SUM[2 + R1] - SUM[3 + R1] = 0\}
\\
\end{mma}
To show that the same recurrence can be
computed also for $S_4(R_1,R_1,0)$,
in view of
\begin{align*}
S_4(R_1,R_2,0)
&=
\sum_{n_1,n_2,n_4,n_5\geq 0} 
q^{Q(R_1,R_2,0,n_1,n_2,n_4, n_5)}
\qbin{0}{n_4}{q}\qbin{n_4}{n_1}{q}
\qbin{R_2}{n_2}{q}\qbin{R_1-n_2}{n_5-n_2}{q}\\
&=
\sum_{n_2,n_5\geq 0} 
q^{Q(R_1,R_2,0,0,n_2, 0, n_5)}
\qbin{R_2}{n_2}{q}\qbin{R_1-n_2}{n_5-n_2}{q},
\end{align*}
it will be convenient to define
\begin{mma}
\In
S4r3ZeroSd[R1\_, R2\_, n2\_, n5\_] := 
q^{Q[0, n2, 0, n5, 0, R2, R1]} 
  qBinomial[R2, n2, q]
 \newline \hspace*{5.5cm}   
   \ast qBinomial[R1 - n2, n5 - n2, q]
\\
\end{mma}
Now we are ready to compute a recurrence
for the double sum $S_4(R_1,R_1,0)$,
\begin{mma}
\In
qFindRecurrence[
  S4r3ZeroSd[R1, R1, n2, n5], 
  \{R1\}, \{n2, n5\}, \{1\}, \{1, 1\},
 \newline \hspace*{5.5cm}     
  \{1, 1\}] // qSumRecurrence[\#, 1] \&
\\
\Out
\{-q^{14 + 9 R1} (-1 + q^{1 + R1})
 (-1 + q^{2 + R1}) SUM[R1] - 
   q^{12 + 6 R1} (-1 + q^{2 + R1})
    (1 + q^{2 + R1}) SUM[1 + R1] 
\newline \hspace*{0.5cm}    
  + q^{8 + 3 R1} (1 + 2 q^{2 + R1}) SUM[2 + R1] - SUM[3 + R1] = 0\}
\\
\end{mma}
Inspection shows that the two recurrences 
of order $3$ coincide. Hence showing the equality,
\[
S_2(R_1,R_1,0)=S_4(R_1,R_1,0), 
\hspace*{0,3cm} R_1 \geq 0,
\]
is reduced to showing the equality at
the three initial values; i.e.,
\begin{equation*}
S_2(R_1,R_1,0)=S_4(R_1,R_1,0)
\text{\, for\, } R_1\in \{0,1,2\}.
\end{equation*}
This task is elementary.

To prove $S_2(R_1,R_1-1,0)=S_4(R_1,R_1-1,0)$ we
compute a common recurrence on the subdiagonal
as follows,
\begin{mma}
\In
(qFindRecurrence[
  S2Summand[R1, R1-1, 0, m1, m2], 
  \{r1\}, \{m1, m2\}, \{1\}, \{1, 1\},
\newline \hspace*{5.3cm}  
   \{1,1\}] // 
 qSumRecurrence[\#, 1] \&)[[1]]
\\
\Out
-q^{10 + 9 R1} (-1 + q^{R1}) (-1 + q^{1 + R1}) SUM[R1] - 
  q^{10 + 6 R1} (-1 + q^{1 + R1}) (1 + q^{1 + R1}) SUM[1 + R1] 
\newline \hspace*{0.3cm}  
  + 
  q^{6 + 3 R1} (1 + 2 q^{2 + R1}) SUM[2 + R1] - SUM[3 + R1] = 0
\\
\end{mma}
and
\begin{mma}
\In
(qFindRecurrence[
    S4r3ZeroSd[R1, R1 - 1, n2, n5], 
    \{R1\}, \{n2, n5\}, \{1\}, \{1, 1\},
\newline \hspace*{5.3cm}     
    \{1, 1\}] // qSumRecurrence[\#, 1] \&)[[1]]
\\
\Out
-q^{10 + 9 R1} (-1 + q^{R1}) (-1 + q^{1 + R1}) SUM[R1] - 
  q^{10 + 6 R1} (-1 + q^{1 + R1}) (1 + q^{1 + R1}) SUM[1 + R1] 
\newline \hspace*{0.3cm}  
  + 
  q^{6 + 3 R1} (1 + 2 q^{2 + R1}) SUM[2 + R1] - SUM[3 + R1] = 0
\\
\end{mma}
Again both recurrences coincide and are of
order $3$. Hence showing,
\begin{equation*}
S_2(R_1,R_1-1,0)=S_4(R_1,R_1-1,0)
\text{\, for\, } R_1\in \{1,2,3\},
\end{equation*}
completes the proof of $S_2(R_1,R_1-1,0)=S_4(R_1,R_1-1,0)$ for
$R_1\geq 1$.

\subsubsection{Proof of Proposition~\ref{prop:Prop1}}
\label{subsec:Prop1Proof}

To prove that the $S_2$-recurrence \texttt{Prec}
from \texttt{Out[\ref{mma:Prec}]} is satisfied
also by $S_4(r_1,r_2,r_3)$, we shall represent recurrences
as $q$-shift operators. Recall that in order
to use the \texttt{HolonomicFunctions} package
in our context, in \texttt{In[\ref{mma:alg}]} we 
specified the non-commutative algebra of $q$-shift operators,
\[
\mathrm{alg}=
\KK(q, qr_1, qr_2, qr_3)[
S_{qr_1, q}, S_{qr_2, q}, S_{qr_3, q}],
\]
with relation~\eqref{eq:non-comm} as the rule
of non-commutativity and writing ``$\dast$'' for
non-commutative multiplication, if required by
the package. In our setting, $\KK=\Q$.

Recall, the recurrence \texttt{P2rec} from \texttt{Out[\ref{mma:P2rec}]} is also satisfied
by $S_4(r_1,r_2,r_3)$. We determine its operator representation as we did for \texttt{P1rec} in
\texttt{Out[\ref{mma:P1}]},
\begin{mma}
\In
P_2 = ToOrePolynomial[
q^{2 r_1+2} {QS}[{qr}_1,q^{r_1}]\dast
   {QS}[{qr}_2,q^{r_2}]-
   q^{2 r_2}
   {QS}[{qr}_1,q^{r_1}]^2+
   q^{3r_1+2 r_2+3}, alg]
\\
\Out
q^2 \text{qr}_1^2 S_{\text{qr}_1,q}
   S_{\text{qr}_2,q}-\text{qr}_2^2
   S_{\text{qr}_1,q}^2+q^3 \text{qr}_2^2 \text{qr}_1^3
\label{mma:P2}\\
\end{mma}

Computational experiments show that we will need 
another recurrence satisfied by $S_4(r_1,r_2,r_3)$,
namely,
\begin{mma}
\In
P3rec =
\newline \hspace*{0.3cm}
(qFindRecurrence[q\ast
    S4Summand[r1, r2, r3, n1, n2, n4, n5], 
    \{r1, r2, r3\}, \{n1, n2, n4, n5\}, 
\newline \hspace*{3.5cm}    
    \{0, 0, 1\}, \{1, 1, 1, 0\}, \{0, 1, 1, 1\}] 
    // qSumRecurrence[\#, 3] \&)[[1]]
\\
\Out
-q^{1 + 2 r1 + r2 + r3} SUM[r1, r2, 1 + r3] + 
  q^{ r2 + 2 r3} (-1 + q^{1 + r2 + r3}) SUM[1 + r1, r2, r3] + 
  \newline \hspace*{0.5cm}    
  q^{-1+r2} SUM[1 + r1, r2, 1 + r3] - 
  q^{3 r3} SUM[1 + r1, 1 + r2, r3] = 0
\\
\end{mma}
We convert this recurrence into operator
\footnote{In order to avoid negative integers
in the $q$-powers, we multiply each coefficient
of the operator
with $q$.} form,
\begin{mma}
\In
P_3 = 
ToOrePolynomial[
-q^{2 + 2 r_1 + r_2 + r_3} 
QS[{qr}_3, q^{r_3}] + 
   q^{1 + r_2 + 2 r_3} (-1 + q^{1 + r_2 + r_3})
    QS[{qr}_1, q^{r_1}] 
   \newline \hspace*{1.5cm}    
    + 
   q^{r_2} QS[{qr}_1, q^{r_1}] \dast 
   QS[{qr}_3, q^{r_3}] - 
   q^{1 + 3 r_3} QS[{qr}_1, q^{r_1}] \dast 
   QS[{qr}_2, q^{r_2}], alg]
\\   
\Out
-q
   \text{qr}_3^3
   S_{\text{qr}_1,q}
   S_{\text{qr}_2,q}+
   \text{qr}_2S_{\text{qr}_1,q}S_{\text{qr}_3,q}+\left(-q
   \text{qr}_2 \text{qr}_3^2+q^2 
   \text{qr}_2^2
   \text{qr}_3^3\right)
   S_{\text{qr}_1,q}-q^2
   \text{qr}_1^2 \text{qr}_2
   \text{qr}_3S_{\text{qr}_3,q}
\\
\end{mma} 

Finally, again using the procedure call
\texttt{ToOrePolynomial}
we convert the recurrence \texttt{Prec} into operator form and obtain
\begin{mma}
\In
P={qr}_3^2 {S}_{{qr}_1,q} {S}_{{qr}_2,q}
-q^5
   {qr}_1^2
   {qr}_3^2 S_{{qr}_1,q} 
   S_{{qr}_2,q}-q^6
   {qr}_1^3
   {qr}_3^2 S_{{qr}_1,q} S_{{qr}_2,q}-
   q^7
   {qr}_1^5 {qr}_2 S_{{qr}_3,q}
\label{mma:P}\\
\end{mma}

Let $R$ be the subalgebra of $\mathrm{alg}$
containing all $q$-shift operators annihilating
$S_4(r_1,r_2,r_3)$. The operators $P_2$ and 
$P_3$ generate a left ideal $I$ of $R$,  
\[
I:= \langle P_2, P_3 \rangle
=
\{ \alpha \dast P_2 + \beta\dast P_3\mid
\alpha, \beta\in \mathrm{alg}\} \subseteq R.
\]
Our goal is to show that $P\in R$.
To this end, with Koutschan's \texttt{HolonomicFunctions} package we compute
a Gr\"obner basis $G=\{g_1,\dots,g_n\}$ for the
ideal $I$; i.e., 
\[
I =\langle g_1,\dots, g_n \rangle.
\]
The decisive consequence of such a $G$ is this:
if $P\in R$ then
\[
P=c_1\dast g_1 +\dots + c_n\dast g_n
\text{\, for certain\, }
c_1, \dots, c_n \in \mathrm{alg};
\]
moreover, these operators $c_j$ can be determined
algorithmically. 

With \texttt{HolonomicFunctions} one computes 
$G$ as follows,
\begin{mma}
\In
G = OreGroebnerBasis[\{P_2,P_3\}]
\\
\Out
\big\{
\left(q^6 \text{qr}_2^2 \text{qr}_3^5-q^7
   \text{qr}_2^3 \text{qr}_3^6\right)
   S_{\text{qr}_1,q}+\text{qr}_1 \text{qr}_2^2
   S_{\text{qr}_3,q}^2 
\newline \hspace*{0.5cm}   
   +\left(q^8 \text{qr}_1 \text{qr}_2^2 
   \text{qr}_3^7-q^8
   \text{qr}_1^2 \text{qr}_2 \text{qr}_3^7-q^7
   \text{qr}_1 \text{qr}_2^2 \text{qr}_3^6+q^7
   \text{qr}_1^2 \text{qr}_2 \text{qr}_3^6-q^6
   \text{qr}_1^2 \text{qr}_3^5+q^6 \text{qr}_1
   \text{qr}_2 \text{qr}_3^5\right) 
   S_{\text{qr}_2,q}
 \newline \hspace*{0.5cm}   
   +  
 \big(-q^9 \text{qr}_1^2
 \text{qr}_2^3
   \text{qr}_3^6+q^8 \text{qr}_1^2 \text{qr}_2^3
   \text{qr}_3^5-q^7 \text{qr}_1 \text{qr}_2^3
   \text{qr}_3^4-q^6 \text{qr}_1 \text{qr}_2^3
   \text{qr}_3^4+q^6 \text{qr}_1 \text{qr}_2^3
   \text{qr}_3^3+q^5 \text{qr}_1 \text{qr}_2^3
   \text{qr}_3^3
  \newline \hspace*{1.2cm} 
   -q^5 \text{qr}_1 \text{qr}_2^2
   \text{qr}_3^2 
   -q^4 \text{qr}_1 \text{qr}_2^2
   \text{qr}_3^2-q^3 \text{qr}_1 \text{qr}_2^2
   \text{qr}_3^2\big) S_{\text{qr}_3,q}
 \newline \hspace*{0.5cm}   
   -  
  q^9 \text{qr}_1 \text{qr}_2^4
   \text{qr}_3^7+q^8 \text{qr}_1 \text{qr}_2^4
   \text{qr}_3^6+q^8 \text{qr}_1 \text{qr}_2^3
   \text{qr}_3^6-2 q^7 \text{qr}_1 \text{qr}_2^3
   \text{qr}_3^5-q^6 \text{qr}_1 \text{qr}_2^3
   \text{qr}_3^5+q^6 \text{qr}_1 \text{qr}_2^2
   \text{qr}_3^4+q^5 \text{qr}_1 \text{qr}_2^2
   \text{qr}_3^4,   
 \newline \hspace*{0.3cm} 
 -q^2 \text{qr}_3 \text{qr}_2^2 S_{\text{qr}_3,q}-q
   \text{qr}_3^2 S_{\text{qr}_2,q}+S_{\text{qr}_2,q}
   S_{\text{qr}_3,q},  
 \newline \hspace*{0.3cm} 
\left(-q^4 \text{qr}_1 \text{qr}_2^3 \text{qr}_3^5
-q^4
   \text{qr}_1 \text{qr}_2^3 \text{qr}_3^4-q^2
   \text{qr}_1 \text{qr}_2^2 \text{qr}_3^3+q^2
   \text{qr}_1^2 \text{qr}_2 \text{qr}_3^3\right)
   S_{\text{qr}_2,q}
   +\left(q^3 \text{qr}_2^4
   \text{qr}_3^4-q^2 \text{qr}_2^3
   \text{qr}_3^3\right) S_{\text{qr}_1,q}
  \newline \hspace*{0.6cm}   
   +\left(q
   \text{qr}_1 \text{qr}_2 \text{qr}_3^4-\text{qr}_1^2
   \text{qr}_3^4\right) S_{\text{qr}_2,q}^2+q
   \text{qr}_1 \text{qr}_2^3 S_{\text{qr}_3,q}+q^5
   \text{qr}_1 \text{qr}_2^5 \text{qr}_3^5-q^4
   \text{qr}_1 \text{qr}_2^4 \text{qr}_3^4+q^3
   \text{qr}_1 \text{qr}_2^4 \text{qr}_3^3-q^2
   \text{qr}_1 \text{qr}_2^3 \text{qr}_3^2,
   \newline \hspace*{0.3cm}     
   \left(-q^2 \text{qr}_1^2 \text{qr}_3^2
   \text{qr}_2^2-\text{qr}_1 \text{qr}_2^2\right)
   S_{\text{qr}_3,q}+\text{qr}_2^2 \text{qr}_3
   S_{\text{qr}_1,q} S_{\text{qr}_3,q}+\left(q
   \text{qr}_1 \text{qr}_2 \text{qr}_3^3-q
   \text{qr}_1^2 \text{qr}_3^3\right)
   S_{\text{qr}_2,q}-q^2 \text{qr}_1 \text{qr}_2^3
   \text{qr}_3^3+q \text{qr}_1 \text{qr}_2^2
   \text{qr}_3^2,
 \newline \hspace*{0.3cm}
\left(q^2 \text{qr}_2^3 \text{qr}_3^4-q \text{qr}_2^2
   \text{qr}_3^3\right) S_{\text{qr}_1,q}-q
   \text{qr}_2 \text{qr}_3^4 S_{\text{qr}_1,q}
   S_{\text{qr}_2,q}+\left(q \text{qr}_1^2
   \text{qr}_3^3-q \text{qr}_1 \text{qr}_2
   \text{qr}_3^3\right) S_{\text{qr}_2,q}+\text{qr}_1
   \text{qr}_2^2 S_{\text{qr}_3,q}
  \newline \hspace*{0.6cm}  
   +q^2 \text{qr}_1
   \text{qr}_2^3 \text{qr}_3^3-q \text{qr}_1
   \text{qr}_2^2 \text{qr}_3^2,    
 \newline \hspace*{0.3cm}
 \left(q^2 \text{qr}_1^3 \text{qr}_2 \text{qr}_3^3-q^2
   \text{qr}_1^4 \text{qr}_3^3\right)
   S_{\text{qr}_2,q}+\left(q^2 \text{qr}_1^2
   \text{qr}_2^2 \text{qr}_3^3-q^3 \text{qr}_1^2
   \text{qr}_2^3 \text{qr}_3^4\right)
   S_{\text{qr}_1,q}+\text{qr}_2^3 \text{qr}_3^4
   S_{\text{qr}_1,q}^2-q \text{qr}_1^3 \text{qr}_2^2
   S_{\text{qr}_3,q}
  \newline \hspace*{0.6cm}     
   -q^3 \text{qr}_1^3 \text{qr}_2^3
   \text{qr}_3^4-q^3 \text{qr}_1^3 \text{qr}_2^3
   \text{qr}_3^3+q^2 \text{qr}_1^3 \text{qr}_2^2
   \text{qr}_3^2 \big\}     
\\
\In
\{g_1,g_2,g_3,g_4,g_5,g_6\}=G;
\label{mma:G}
\\
\end{mma}

We see that the Gr\"obner basis $G$ consists
of $6$ elements. Our proof is completed with
the following \texttt{HolonomicFunctions}
procedure call; a description
of its functionality can be obtained as follows,
\begin{mma}
\In
?OreReduce
\\
\Print
OreReduce[p, \{g1, g2, ...\}] reduces the OrePolynomial p with the \
Groebner basis given by \{g1, g2, ...\}. 
\newline
The option ``Extended $\rightarrow$ True''
computes also the cofactors; i.e., the output 
is 
\newline
\{r, f, \{c1, c2, $\dots$\}\} such that 
f $\dast$ p = r + c1 $\dast$ g1 
+ c2 $\dast$ g2 + $\dots$ . 
\\ \label{mma:OreRed}
\end{mma} 
Finally, we are ready to reduce $P$ from 
\texttt{In[\ref{mma:P}]} with respect to
the Gr\"obner basis $G$  from 
\texttt{In[\ref{mma:G}]},
\begin{mma}
\In
\{r,f,\{c_1,c_2,c_3,c_4,c_5,c_6\}\}=
OreReduce[P, \{g_1,g_2,g_3,g_4,g_5,g_6\}, 
Extended \rightarrow True] // Factor
\\
\Out
\Big\{ 0, 1,
\Big\{
0,0,0,
\frac{q \text{qr1} }{\text{qr2}
   \text{qr3}^3}S_{\text{qr1},q}+\frac{q^5 \text{qr1}^3}{\text{qr2}
   \text{qr3}^2},
\newline \hspace*{1.2cm}
 -\frac{1}{q \text{qr2} \text{qr3}^2}  S_{\text{qr1},q}^2
 +
 \frac{q \text{qr1} \left(q^3
   \text{qr1} \text{qr2} \text{qr3}-q^2
   \text{qr1}+\text{qr2}\right)}{\text{qr2}^2
   \text{qr3}^3}S_{\text{qr1},q} 
   +
 \frac{q^5 \text{qr1}^3}{\text{qr2} \text{qr3}^2},
\frac{(q \text{qr2} \text{qr3}-1)
   }{\text{qr2}^2 \text{qr3}^3} S_{\text{qr1},q}  
\Big\}\Big\}
\\
\end{mma}
Matching this output to the 
description from \texttt{In[\ref{mma:OreRed}]}
we obtained $(r,f)=(0,1)$ and thus,
\begin{equation}\label{eq:Pcomb}
P= c_1 \dast g_1 + \dots + c_6 \dast g_6.
\end{equation}
In other words, $P\in R$, which completes
the proof of Proposition~\ref{prop:Prop1} 
and thus Lemma~\ref{lem:4} is proven.

\begin{remark}
If $P\not\in R$ then
the algorithm would detect this with $r\neq 0$, and we would
need to enrich the ideal $I$ with further
operators annihilating $S_2(r_1,r_2,r_3)$.
Following our proof strategy,
these operators, as $P_2$ and $P_3$, would 
be computed using the \texttt{qMultiSum}
package.
\end{remark}

\begin{remark}
The second element of the Gr\"obner basis $G$
is
\begin{mma}
\In 
g_2
\\
\Out
-q^2 \text{qr}_3 \text{qr}_2^2 S_{\text{qr}_3,q}-q
   \text{qr}_3^2 S_{\text{qr}_2,q}+S_{\text{qr}_2,q}
   S_{\text{qr}_3,q}
\\
\end{mma}
Recalling \texttt{Out[\ref{mma:P1}]},
this is the operator form $P_1$ of
the recurrence
\texttt{P1rec} which we computed with \texttt{qMultiSum}. In this
case, being 
a member of the Gr\"obner basis $G$ of $I$,
adding $P_1$ to the
set of generators $\{P_2, P_3\}$ would not
change the ideal $I=\langle P_2, P_3\rangle=
\langle P_1, P_2, P_3\rangle$. In other words,
$q$-shift operators computed with \texttt{qMultiSum}
not necessarily enrich the ideal under consideration.
\end{remark}

\subsubsection{Independent Verification of
the Proof of Proposition~\ref{prop:Prop1}}
\label{subsec:Prop1ProofIndep}

Remarkably, the algorithms used
in our computer-assisted proof of Lemma~\ref{lem:4}
are designed in a way which allows 
verifications of the
computed results \textit{independently from the
steps in which implementations of these 
algorithms have computed them}.

Independent verification of the sum recurrences
computed by the \texttt{qMultiSum} package can be
done by direct verification of the underlying (certificate)
recurrences for the summands which also were
computed with \texttt{qMultiSum}.  
This task boils down
to zero recognition of multivariate rational
functions. For further details
see~\cite{WZ} and~\cite{Riese}.

Independent verification of our Gr\"obner basis
proof of Proposition~\ref{prop:Prop1} is also
possible. To this end, one invokes Koutschan's
Gr\"obner bases call \texttt{OreGroebnerBasis}
together with the
option ``\texttt{Extended $\rightarrow$ True}'',
\begin{mma}
\In
\{G, M\} 
= 
 OreGroebnerBasis[\{P_2, P_3\}, 
  Extended \rightarrow True] //Simplify;
\\
\end{mma}
The output is the Gr\"obner basis $G=\{g_1,\dots,
g_6\}$ from \texttt{Out[\ref{mma:G}]} 
together with a $6\times 2$ matrix
$M=(m_{i,j})$, $1\leq i\leq 6$, $1\leq j\leq 2$,
of $q$-shift operators. These $m_{i,j}\in \mathrm{alg}$, called the co-factors, have
the property,
\begin{align*}
g_1&=m_{1,1}\dast P_2 + m_{1,2}\dast P_3,\\
g_2&=m_{2,1}\dast P_2 + m_{2,2}\dast P_3,\\
&\dots\\
g_6&=m_{6,1}\dast P_2 + m_{6,2}\dast P_3.
\end{align*}
Consequently, for the operator
$P$ from~\eqref{eq:Pcomb} we have,
\begin{align}
P &= \sum_{k=1}^6 c_k\dast g_k 
=
\sum_{k=1}^6 c_k\dast 
(m_{k,1}\dast P_2 + m_{k,2}\dast P_3) 
\nonumber\\
&=
\underbrace{\sum_{k=1}^6 (c_k\dast 
m_{k,1})}_{=: C_2}\dast P_2 +
\underbrace{\sum_{k=1}^6 (c_k\dast 
m_{k,2})}_{=: C_3}\dast P_3.
\label{eq:PindCheck}
\end{align}
Summarizing, to prove Proposition~\ref{prop:Prop1}
for \texttt{Prec},
we need to compute the operator forms
$P, P_2$ and $P_3$ of the respective
recurrences.
These recurrences are obtained with \texttt{qMultiSum}; independent 
verification of the correctness
of these recurrences is done on the level
of summand recurrences, also delivered
by \texttt{qMultiSum}. To prove that $P$ annihilates
also $S_4(r_1,r_2,r_3)$, it suffices
to verify relation~\eqref{eq:PindCheck}. This
verification is a routine exercise in the
algebra $\mathrm{alg}$ of $q$-shift operators,
once the constituents $c_k$ and $m_{i,j}$ are 
delivered by the \texttt{HolonomicFunctions}
package as above. We stress again that  
for the verification of~\eqref{eq:PindCheck}
one only needs the explicit forms of the
$c_k$ and $m_{i,j}$; in other words, it is
irrelevant in which way they have been obtained.

We conclude this section with the 
\texttt{HolonomicFunctions} procedure call
to execute the independent verification
of $P\in R$, an elementary computation
in the operator algebra,
\begin{mma}
\In
P-\left(\sum_{k=1}^6 c_k\dast M[[k,1]]\right)
\dast P_2
-
\left(\sum_{k=1}^6 c_k\dast M[[k,2]]\right)
\dast P_3 \hspace*{0.3cm}//Simplify
\label{mma:PindCheck}\\
\Out
0
\\
\end{mma}

\section{Conjecture~\ref{conj:sum} implies Conjecture~\ref{conj:HJO}, $a=3$}\label{sec:sum-implies-hjo}
\begin{proof}
    [Proof that Conjecture~\ref{conj:sum} for $b$ implies Conjecture~\ref{conj:HJO} for $(3,b)$]
    Let the left-hand side of \eqref{eq:sum-3k+1} (if $b=3k+1$) or \eqref{eq:sum-3k-1} (if $b=3k-1$) be denoted by $S(\mathbf{r})$. Then it can be observed from the definition that
    \[ Z_{3,b}(q)=\sum_{r_1,\dots,r_k} \frac{1}{(q)_{r_1}}\left(\prod_{i=1}^{k-1}\qbinom{r_i}{r_{i+1}}\right)S(\mathbf{r}).\]
    Replacing $S(\mathbf{r})$ by the right-hand side of \eqref{eq:sum-3k+1}  or \eqref{eq:sum-3k-1}, $Z_{3,b}(q)$ precisely becomes the left-hand side of \cite[(1.11)]{WarnaarAG} with $\mathbf{k}=k+1$ (if $b=3k+1$) or \cite[(1.10)]{WarnaarAG} with $\mathbf{k}=k$ (if $b=3k-1$); here $\mathbf{k}$ denotes $k$ in \cite{WarnaarAG}. Hence $Z_{3,b}(q)$ equals to the corresponding right-hand side of \cite[(1.10) or (1.11)]{WarnaarAG}, which is precisely $P_{3,b}(q)$.
\end{proof}

\section{Appendix: AxiomProver and Formal Certificate}
AxiomProver is an AI system currently under development by Axiom Math. It was used to generate a formal certificate for the results in this paper (i.e. Theorems \ref{thm:main}, \ref{thm:sum}, \ref{thm:q=1}), relative to existing literature (i.e. the results in \cite{WarnaarAG}). The formal certificate can be found in:
\begin{center}
\url{https://github.com/AxiomMath/HJOa3Small}
\end{center}
It contains a formal challenge file containing the statements of the results, which can be mechanically verified using the Comparator tool in Lean.

\end{document}